\documentclass[12pt]{article}
\usepackage{amsmath}
\usepackage{color}
\usepackage{amsfonts}
\usepackage{amsmath, amsthm, amssymb}
\usepackage{latexsym}

\usepackage{epsfig}
\usepackage{color}

\numberwithin{equation}{section}

\newtheorem{lemma}{Lemma}[section]

\newtheorem{example}{Example}[section]

\newtheorem{corollary}{Corollary}[section]

\newtheorem{definition}{Definition}[section]

\newcommand{\be}{\begin{equation}}
\newcommand{\ee}{\end{equation}}
\newcommand{\bea}{\begin{eqnarray}}
\newcommand{\eea}{\end{eqnarray}}
\newcommand{\beas}{\begin{eqnarray*}}
\newcommand{\eeas}{\end{eqnarray*}}
\newcommand{\ba}{\begin{array}}
\newcommand{\ea}{\end{array}}

\def\bge{\begin{eqnarray}}
\def\bgee{\begin{eqnarray*}}
\def\ege{\end{eqnarray}}
\def\egee{\end{eqnarray*}}
\begin{document}

\title{ Fractional Differential Equations Involving Caputo Fractional Derivative with Mittag-Leffler Non-Singular Kernel: Comparison Principles and Applications}

\author{Mohammed Al-Refai\\
{\small Department of Mathematical Sciences, UAE University,}\\ {\small P.O.Box 15551, Al Ain, UAE.}\\
\small{ m${}_{-}$alrefai@uaeu.ac.ae }
 }
 \maketitle

\begin{abstract}

In this paper we study linear and nonlinear  fractional differential equations involving the   Caputo  fractional derivative  with Mittag-Leffler non-singular kernel of order
 $0<\alpha<1.$  We first obtain a new estimate of  the fractional derivative of a function at its extreme points and derive a necessary condition for the existence of a solution to the linear fractional equation. The obtained sufficient condition determine the  initial condition of the associated fractional initial value problem.  We then derive comparison principles to the linear fractional equations. We apply these principles to obtain a norm estimate of  solutions to the linear equation and to obtain a uniqueness result to the nonlinear equation. We also  derive a lower and upper bound of solutions to the nonlinear equation. The applicability of the new results is illustrated through several examples.
\newline
Key words and phrases: Fractional differential  equations,  Maximum principle, Fractional derivatives.
\noindent

\end{abstract}

\section{Introduction}

Fractional models have been implemented to model various problems in several fields, \cite{FDL,Hil00,Kla08,Mai_book}.  The non-locality of the fractional derivative makes fractional models  more practical than the usual ones, especially for  systems which involve memory. In recent years there are great interests to develop new types of non-local fractional derivative with non-singular kernel, see \cite{abdon,caputo}.  The idea is to have more types of non-local fractional derivative, and it is the role of  application that will determine which fractional model is more proper. The theory of fractional models is effected by  the type of the fractional derivative. Therefore, several papers have been devoted recently to study the  new types of fractional derivatives and their applications, see \cite{TD ROMP, TD ADE  Mont} for the Caputo-Fabrizio fractional derivative and \cite{abdon2,TD ADE 2016,badr,badr2,dj,gau} for the Abdon-Baleanu fractional derivative.

In this paper, we analyze the solutions of a class of fractional differential equations involving the Caputo fractional derivative with Mittag-Leffler non-singular kernel of order
 $0<\alpha<1.$ To the best of our knowledge this is the first theoretical study of fractional differential equations with fractional derivative of non-singular kernel.  We start with the definition and main properties of the nonlocal fractional derivative with Mittag-leffler non-singular kernel. For more details the reader is referred to \cite{abdon,abdon2,thabet1}.

 \begin{definition} \label{DEF1}
Let $f \in H^1(a,b),~~a<b,~~\alpha \in (0,1)$, the left Caputo fractional derivative with Mittag-Leffler non-singular kernel  is defined by
\begin{equation}\label{d1}
 ({}^{ABC}  {}_{a}D^\alpha f)(t)=\frac{B(\alpha)}{1-\alpha} \int_a^t E_\alpha{\left[-\frac{\alpha}{1-\alpha}(t-s)^\alpha\right]} f^\prime(s) ds.
\end{equation}
where $B(\alpha)>0$ is a normalization function  satisfying $B(0)=B(1)=1,$ and $E_\alpha(s)$ is the well known Mittag-Leffler function..
\begin{definition}
Let $f \in H^1(a,b),~~a<b,~~\alpha \in (0,1)$, the left Riemann-Liouville fractional derivative with Mittag-Leffler non-singular kernel  is defined by
\begin{equation}\label{d1}
 ({}^{ABR} {}_{a}D^\alpha f)(t)=\frac{B(\alpha)}{1-\alpha} \frac{d}{dt} \int_a^t E_\alpha{\left[-\frac{\alpha}{1-\alpha} (t-s)^\alpha \right]} f(s) ds.
\end{equation}
\end{definition}
The associated fractional integral is defined by
\begin{equation}\label{d3}
  ({}^{AB}{}_{a}I^\alpha f)(t)=\frac{1-\alpha}{B(\alpha)}f(t)+\frac{\alpha}{B(\alpha)}(_{a}I^\alpha f)(t),
\end{equation}
\end{definition}
where $(_{a}I^\alpha f)(t)$ is the left Riemann-Liouville fractional integral of order $\alpha> 0$  defined by
$$(_{a}I^\alpha f)(t)=\frac{1}{\Gamma(\alpha)}\int_a^t (t-s)^{\alpha-1} f(s) ds.$$
The following holds true,
\begin{eqnarray}
({}^{ABC} {}_{0}D^\alpha f)(t)&=&({}^{ABR} {}_{0}D^\alpha f)(t)-\frac{B(\alpha)}{1-\alpha} f(0) E_\alpha(-\frac{\alpha}{1-\alpha} t^\alpha),\label{prop1}\\
({}^{ABR} {}_{a}D^\alpha  \ {}^{AB} {}_{a}I^\alpha f)(t)&=&f(t),\label{prop2}\\
({}^{AB} {}_{a}I^\alpha \ {}^{ABR} {}_{a}D^\alpha    f)(t)&=&f(t).\label{prop3}
\end{eqnarray}

The rest of the  paper is organized as follows. In Section 2, we present a  new estimate of the fractional derivative of a function at its extreme points. In Section 3, we develop new comparison principles for linear fractional equations and obtain a norm bound to their solutions. We also, obtain the solution solution for a class of linear equations in a closed form,  and present a    necessary condition for the existence of their solutions. In Section 4, we consider nonlinear fractional equations. We  obtain a uniqueness result and derive upper and lower bounds to the solution of the problem. Finally we present some examples to illustrate the applicability of the obtained results.

\section{Estimates of fractional derivatives at extreme points}

We start with estimating the fractional derivative of a function at its extreme points, this result is analogous  to the ones obtained in \cite{ref10} for the Caputo and Riemann-Liouville fractional derivatives. The applicability of these results were indicated in (\cite{refai2}-\cite{refai-distributed}) by establishing new comparison principles and studying various fractional diffusion models. Therefore, the current result can be used to study fractional diffusion models involving the Caputo and Riemann-Liouville fractional derivatives with Mittag-Leffler non-singular kernel, and we leave this for a future work.

\begin{lemma}
\label{th1}
Let a function $f\in H^1(a,b)$ attain its maximum at a point $t_0\in [a,b]$ and $0<\alpha <1$. Then the inequality
\begin{equation}
\label{ineq}
({}^{ABC} {}_{a}D^\alpha f)(t_0)\ge \frac{B(\alpha)}{1-\alpha}\ E_\alpha[{-\frac{\alpha}{1-\alpha}(t_0-a)^\alpha}] (f(t_0)-f(a))\ge 0.
\end{equation}
holds true.
\end{lemma}
\begin{proof}
We define the auxiliary function \ $g(t)=f(t_0)-f(t), \ t\in [a,b].$ Then it follows that $g(t)\ge 0,$ on $[a,b],$  $g(t_0)=g'(t_0)=0$ and  $({}^{ABC} {}_{a}D^\alpha g)(t)=-({}^{ABC}{}_{a}D^\alpha f)(t).$ Since $g\in H^1(a,b)$, then $g'$ is integrable and
integrating by parts with \ $$u=E_\alpha[{-\frac{\alpha}{1-\alpha}(t_0-s)^\alpha}], \
\mbox{and} \ dv=g'(s) ds,$$ yields

\begin{eqnarray}
({}^{ABC} {}_{a}D^\alpha g)(t_0)&=&\frac{B(\alpha)}{1-\alpha}\int_a^{t_0}  E_\alpha[{-\frac{\alpha}{1-\alpha}(t_0-s)^\alpha}]g'(s) \ ds\nonumber\\
&=&\frac{B(\alpha)}{1-\alpha}\bigg(E_\alpha[{-\frac{\alpha}{1-\alpha}(t_0-s)^\alpha}]  g(s)|_a^{t_0}- \int_a^{t_0} \frac{d}{ds} E_\alpha[{-\frac{\alpha}{1-\alpha}(t_0-s)^\alpha}]g(s) ds\bigg) \nonumber\\
&=&\frac{B(\alpha)}{1-\alpha}\bigg(E_\alpha[0] g(t_0)-E_\alpha[{-\frac{\alpha}{1-\alpha} (t_0-a)^\alpha}] g(a)- \int_a^{t_0} \frac{d}{ds} E_\alpha[{-\frac{\alpha}{1-\alpha}(t_0-s)^\alpha}] g(s) ds\bigg)\nonumber\\
&=&\frac{B(\alpha)}{1-\alpha}\bigg(-E_\alpha[{-\frac{\alpha}{1-\alpha} (t_0-a)^\alpha}] g(a)- \int_a^{t_0} \frac{d}{ds} E_\alpha[{-\frac{\alpha}{1-\alpha}(t_0-s)^\alpha}] g(s) ds\bigg).\label{equ1}
\end{eqnarray}
We recall that for $0<\alpha<1$, see \cite{gorenflo}, we have
$$E_\alpha(-t^\alpha)=\int_0^\infty e^{-rt} K_\alpha(r) dr,$$
where $$K_\alpha(r)=\frac{1}{\pi} \frac{r^{\alpha-1} \sin(\alpha \pi)}{r^{2\alpha}+2r^{\alpha}\cos(\alpha \pi)+1}>0.$$ Thus,
\begin{eqnarray}
\frac{d}{ds} E_\alpha[-\frac{\alpha}{1-\alpha}(t_0-s)^\alpha]&=&\frac{d}{ds} E_\alpha\bigg[-\bigg((\frac{\alpha}{1-\alpha})^{1/\alpha}(t_0-s)\bigg)^\alpha\bigg]\nonumber\\
&=&\frac{d}{ds}\int_0^\infty e^{-r(\frac{\alpha}{1-\alpha})^{1/\alpha}(t_0-s)}K_\alpha(r)dr=\int_0^\infty \frac{d}{ds} e^{-r(\frac{\alpha}{1-\alpha})^{1/\alpha}(t_0-s)}K_\alpha(r)dr\nonumber\\
&=& (\frac{\alpha}{1-\alpha})^{1/\alpha}\int_0^\infty r e^{-r(\frac{\alpha}{1-\alpha})^{1/\alpha}(t_0-s)}K_\alpha(r)dr>0,
\end{eqnarray}
  which together with
 $g(t)\ge 0$ on $[a,b]$, will lead to  the integral in Eq. (\ref{equ1}) is nonnegative. We recall here that $E_\alpha[t]>0, \  0<\alpha<1,$ see \cite{hannaken}, and thus
\begin{eqnarray}
({}^{ABC} {}_{a}D^\alpha g)(t_0)&\le&\frac{B(\alpha)}{1-\alpha}\bigg(-E_\alpha[{-\frac{\alpha}{1-\alpha} (t_0-a)^\alpha}] g(a)\bigg)\nonumber\\
&&= -\frac{B(\alpha)}{1-\alpha}\ E_\alpha[{-\frac{\alpha}{1-\alpha} (t_0-a)^\alpha}] ( f(t_0)-f(a))\le 0 .\label{g11}
\end{eqnarray}
The last inequality yields
$$-({}^{ABC} {}_{a}D^\alpha f)(t_0) \le -\frac{B(\alpha)}{1-\alpha}\ E_\alpha[{-\frac{\alpha}{1-\alpha} (t_0-a)^\alpha}] ( f(t_0)-f(a))\le 0,$$ which proves the result.
\end{proof}

By applying analogous steps to $-f$ we have

\begin{lemma}
\label{th12}
Let a function $f\in H^1(a,b)$ attain its minimum at a point $t_0\in [a,b]$ and $0<\alpha <1$. Then the inequality
\begin{equation}
\label{ineq}
({}^{ABC} {}_{a}D^\alpha f)(t_0)\le \frac{B(\alpha)}{1-\alpha}\ E_\alpha[{-\frac{\alpha}{1-\alpha}t_0}] (f(t_0)-f(a))\le 0.
\end{equation}
holds true.
\end{lemma}
\begin{lemma}
\label{zero}
Let a  function $f\in H^1(a,b)$ then it holds that
\begin{equation}
\label{ineq1}
({}^{ABC} {}_{a}D^\alpha f)(a)=0, \ \ 0<\alpha<1.
\end{equation}

\end{lemma}
\begin{proof}

Because $E_\alpha[-\frac{\alpha}{1-\alpha} (t-s)]$ is continuous on $[a,b]$, then it is in $L^2[a,b].$ Applying the Cauchy-Schwartz inequality we have
\begin{eqnarray}
|({}^{ABC} {}_{a}D^\alpha f)(t)|^2&\le &\frac{B^2(\alpha)}{(1-\alpha)^2}\int_a^{t}  \bigg(E_\alpha[{-\frac{\alpha}{1-\alpha}(t-s)^\alpha}]\bigg)^2 \ ds \  \int_a^t \bigg(f'(s)\bigg)^2 \ ds.\label{zzz}
\end{eqnarray}
Since $f \in H^1(a,b)$ then $f'$ is square integrable and it holds that $\int_a^a \big(f'(s)\big)^2 \ ds=0.$ The result is obtained as the first integral in Eq. (\ref{zzz}) is bounded.
\end{proof}


\section{The linear  equation}
We implement the results in Section 1 to obtain new comparison principles of the linear fractional differential equations of order $0<\alpha<1,$ and to derive a necessary condition for the existence of their solutions.  We then use these principles to obtain a norm
 bound of the solution. We also present the solution of certain linear equation by the Laplace transform.

\begin{lemma}(Comparison Principle-1)\label{comp1}
Let a function $u\in H^1(a,b)\cap C[a,b]$ satisfies the fractional inequality
\begin{equation}\label{ineq1}
P_\alpha(u)=({}^{ABC} {}_{a}D^\alpha u)(t)+p(t) u(t) \le 0, \ t>a, \ 0<\alpha<1,
\end{equation}
where $p(t)\ge  0$  is continuous on $[a,b]$ and $p(a)\neq 0$. Then $u(t)\le 0, \ t\ge a.$
\end{lemma}

\begin{proof}
Since $u\in H^1(a,b)$ then by Lemma \ref{zero} we have $({}^{ABC} {}_{a}D^\alpha u)(a)=0.$ By the continuity of the solution, the fractional inequality (\ref{ineq1}) yields
$p(a) u(a)\le 0$, and hence $u(a)\le 0.$ Assume by contradiction that the result is not true, because $u$ is continuous on $[a,b]$ then $u$ attains absolute maximum at $t_0\ge a$ with $u(t_0)>0.$
Since $u(a)\le 0,$ then $t_0>a.$ Applying the result of Lemma \ref{th1} we have
$$({}^{ABC} {}_{a}D^\alpha u)(t_0)\ge \frac{B(\alpha)}{1-\alpha}\ E_\alpha[{-\frac{\alpha}{1-\alpha}(t_0-a)^\alpha}] (u(t_0)-u(a))> 0.$$ We have
$$({}^{ABC} {}_{a}D^\alpha u)(t_0)+p(t_0) u(t_0)\ge ({}^{ABC} {}_{a}D^\alpha u)(t_0)>0,$$
which contradicts the fractional inequality (\ref{ineq1}), and completes the proof.

\end{proof}

\begin{corollary}(Comparison Principle-2)
Let  $u_1,u_2\in H^1(a,b)\cap C[a,b]$ be the solutions of
\begin{eqnarray}
({}^{ABC} {}_{a}D^\alpha u_1)(t)+p(t) u_1(t)&=&g_1(t), \ t>a, \ 0<\alpha<1,\nonumber\\
({}^{ABC} {}_{a}D^\alpha u_2)(t)+p(t) u_2(t)&=&g_2(t), \ t>a, \ 0<\alpha<1,\nonumber
\end{eqnarray}
where $p(t)\ge  0, g_1(t), g_2(t)$  are continuous on $[a,b]$ and $p(a)\neq 0.$ If $g_1(t)\le g_2(t),$ then it holds that $$u_1(t)\le u_2(t), \ \ \  t\ge a.$$
\end{corollary}

\begin{proof}
Let $z=u_1-u_2,$ then it holds that
\begin{equation}
P_\alpha(z)=({}^{ABC} {}_{a}D^\alpha z)(t)+p(t) z(t)=g_1(t)-g_2(t) \le 0, \ t>a, \ 0<\alpha<1.
\end{equation}
Applying virtue of Lemma \ref{comp1} we have $z(t)\le 0$, and hence the result.
\end{proof}

\begin{lemma}\label{big}
Let $u \in H^1(a,b)$ be the solution of
\begin{eqnarray}
({}^{ABC} {}_{a}D^\alpha u)(t)+p(t) u(t)&=&g(t), \ t>a, \ 0<\alpha<1,\nonumber
\end{eqnarray}
where $p(t)> 0$  is continuous on $[a,b].$  Then it holds that $$||u||_{[a,b]}=\max_{t\in [a,b]}|u(t)| \le M=\max_{t\in [a,b]}\{|\frac{g(t)}{p(t)}|\}.$$
\end{lemma}
\begin{proof}
We have $M\ge |\frac{g(t)}{p(t)}|,$ or $M p(t) \ge |g(t)|, \ t\in [a,b].$
Let $v_1=u-M,$ then it holds that
\begin{eqnarray}
P_\alpha (v_1)&=& ({}^{ABC} {}_{a}D^\alpha v_1)(t)+p(t) v_1(t)=({}^{ABC} {}_{a}D^\alpha u)(t)+p(t) u(t)- p(t) M\nonumber\\
&=&g(t)-p(t) M \le |g(t)|-p(t) M\le 0.\nonumber
\end{eqnarray}
 Thus by virtue of Lemma \ref{comp1} we have
$
v_1=u-M\le 0$, or
\begin{equation}
u\le M. \label{dd1}
\end{equation}
Analogously, let  $v_2=-M-u$, then it holds that
\begin{eqnarray}
P_\alpha (v_2)&=& ({}^{ABC} {}_{a}D^\alpha v_2)(t)+p(t) v_2(t)=-({}^{ABC} {}_{a}D^\alpha u)(t)-p(t) u(t)- p(t) M\nonumber\\
&=&-g(t)-p(t) M \le -g(t)-|g(t)| \le  0. \nonumber
\end{eqnarray}
 Thus by the result of Lemma \ref{comp1} we have
$
v_2=-u-M\le 0$, or
\begin{equation}
u\ge -M. \label{dd2}
\end{equation}
By combining the results of Eq's (\ref{dd1}) and (\ref{dd2}) we have $|u(t)|\le M, \ t\in [a,b]$ and hence the result.
\end{proof}

\begin{lemma}\label{existence}

The fractional initial value problem
\begin{eqnarray}
({}^{ABC} {}_{a}D^\alpha u)(t)&=&\lambda u+f(t), \ \ t>0, \ 0<\alpha<1, \label{qw1}\\
u(0)&=&u_0. \label{qw2}
\end{eqnarray}
has the unique solution
\begin{equation}\label{sol}
u(t)=\frac{1}{B(\alpha)-\lambda(1-\alpha)}\bigg(B(\alpha) u_0 E_\alpha[\omega t^\alpha]+(1-\alpha)( g(t)*f'(t)+f(0) g(t))\bigg),
\end{equation}
in the functional space $H^1(0,b)\cap C[0,b]$, if and only if, $\lambda u_0+f(0)=0, $ where   $\omega=\frac{\lambda \alpha}{B(\alpha)-\lambda(1-\alpha)}, $ and $$
g(t)=E_\alpha[\omega t^\alpha]+\frac{\alpha}{1-\alpha}\frac{t^{\alpha-1}}{\Gamma(\alpha)}*E_\alpha[wt^\alpha].$$
\end{lemma}
\begin{proof}
Since $u\in H^1(0,b)$ we have  $({}^{ABC} {}_{a}D^\alpha u)(0)=0.$ Thus, a necessary condition for the existence of a solution to Eq. (\ref{qw1}) is that
\begin{equation}\label{nec}
\lambda u_0+f(0)=0.
\end{equation}
Applying the Laplace transform to Eq. (\ref{qw1}) and using the fact that $$({}^{ABC} {}_{0}D^\alpha u)(t)=\frac{B(\alpha)}{1-\alpha} E_\alpha[-\frac{\alpha}{1-\alpha} t^\alpha]*u'(t),$$
we have

\begin{eqnarray}
\lambda L(u)+L(f(t))&=&\frac{B(\alpha)}{1-\alpha} L\bigg(E_\alpha[-\frac{\alpha}{1-\alpha} t^\alpha]*u'(t)\bigg).\nonumber
\end{eqnarray}

Applying the convolution result of the Laplace transform together with
$$ L(E_\alpha[-\frac{\alpha}{1-\alpha} t^\alpha])=\frac{s^{\alpha-1}}{s^\alpha+\frac{\alpha}{1-\alpha}}, \ \ \ \ \ \ |\frac{\alpha}{1-\alpha}\frac{1}{s^\alpha}|<1,  $$
will lead to
\begin{eqnarray}
\lambda L(u)+L(f(t))&=& \frac{B(\alpha)}{1-\alpha} \frac{s^{\alpha-1}}{s^\alpha+\frac{\alpha}{1-\alpha}}(s L(u)-u(0)).
\end{eqnarray}
Direction calculations will lead to
\begin{equation}\label{laplace}
L(u)=\frac{B(\alpha) u_0}{B(\alpha)-\lambda(1-\alpha)}\frac{s^{\alpha-1}}{s^\alpha-\omega}+\frac{1-\alpha}{B(\alpha)-\lambda (1-\alpha)} \frac{s^\alpha+\frac{\alpha}{1-\alpha}}{s^\alpha-\omega} L(f(t)),
\end{equation}
where $\omega=\frac{\lambda \alpha}{B(\alpha)-\lambda (1-\alpha)}.$ Thus,
\begin{eqnarray}
u(t)&=&\frac{B(\alpha) u_0}{B(\alpha)-\lambda(1-\alpha)}L^{-1}\bigg(\frac{s^{\alpha-1}}{s^\alpha-\omega}\bigg)+\frac{1-\alpha}{B(\alpha)-\lambda (1-\alpha)} L^{-1}\bigg( \frac{s^\alpha+\frac{\alpha}{1-\alpha}}{s^\alpha-\omega} L(f(t))\bigg),\nonumber\\
&=&\frac{B(\alpha) u_0}{B(\alpha)-\lambda(1-\alpha)} E_\alpha[\omega t^\alpha]+\frac{1-\alpha}{B(\alpha)-\lambda (1-\alpha)} L^{-1}\bigg( \frac{s^\alpha+\frac{\alpha}{1-\alpha}}{s^\alpha-\omega} L(f(t))\bigg).
\label{laplace2}
\end{eqnarray}
Let $$G(s)=\frac{1}{s}\frac{s^\alpha+\frac{\alpha}{1-\alpha}}{s^\alpha-\omega}=
\frac{s^{\alpha-1}}{s^\alpha-\omega}+\frac{\alpha}{1-\alpha}\frac{1}{s^\alpha} \frac{s^{\alpha-1}}{s^\alpha-\omega},$$
then $$g(t)=L^{-1}\big( G(s)\big)= E_\alpha(\omega t)+\frac{\alpha}{1-\alpha}\frac{t^{\alpha-1}}{\Gamma(\alpha)}*E_\alpha(\omega t^\alpha).$$

Applying the convolution result we have
\begin{eqnarray}
L^{-1}\bigg(\frac{s^\alpha+\frac{\alpha}{1-\alpha}}{s^\alpha-\omega} L(f(t))\bigg)&=&L^{-1}\bigg(G(s) s L(f(t))\bigg)\nonumber\\
&=&L^{-1}\bigg(G(s)[s L(f)-f(0)+f(0)]\bigg)=L^{-1}\bigg(G(s)[L(f')+f(0)]\bigg)\nonumber\\
&=&L^{-1}\bigg(G(s)L(f')+f(0) G(s)\bigg)=g(t)*f'(t)+f(0) g(t).\label{laplace3}
\end{eqnarray}
The result follows by substituting Eq. (\ref{laplace3}) in Eq. (\ref{laplace2}).
\end{proof}

\begin{corollary}\label{uniqueness}
The fractional differential equation
\begin{equation}\label{frac1}
({}^{ABC} {}_{0}D^\alpha u)(t)=\lambda u, \ \ t>0, \ 0<\alpha<1,
\end{equation}
has only the trivial solution $u=0$, in the functional space $H^1(0,b)\cap C[0,b]$.
\end{corollary}
\begin{proof}
Applying the result of Lemma \ref{existence} with $f(t)=0,$ yields
$$u(t)=\frac{1}{B(\alpha)-\lambda(1-\alpha)} B(\alpha) u_0 E_\alpha[\omega t^\alpha].$$
The necessary condition for the existence of solution yields that $u_0=0,$ and hence the result.
\end{proof}

\section{The nonlinear equation}
In this section we apply the obtained comparison principles to establish a uniqueness result for a nonlinear   fractional differential equation and to estimate its solution. We have

\begin{lemma}\label{nonlinear}
Consider the nonlinear fractional differential equation
\begin{equation}\label{frac1}
({}^{ABC} {}_{a}D^\alpha u)(t)=f(t,u), \ \ t>a, \ 0<\alpha<1,
\end{equation}
 where $f(t,u)$ is a smooth function. If $f(t,u)$ is non-increasing with respect to $u$ then the above equation has at most one
  solution $u \in H^1(a,b)$.
\end{lemma}

\begin{proof}

Assume that $u_1,u_2 \in H^1(a,b)$ be two solutions of the above equation and let $z=u_1-u_2.$ Then it holds that
$$({}^{ABC} {}_{a}D^\alpha z)(t)=f(t,u_1)-f(t,u_2).$$ Applying the mean value theorem we have
$$f(t,u_1)-f(t,u_2)=\frac{\partial f}{\partial u}(u^*) (u_1-u_2),$$ for some $u^*$ between $u_1$ and $u_2.$ Thus,
\begin{equation}\label{ll}
({}^{ABC} {}_{a}D^\alpha z)(t)-\frac{\partial f}{\partial u}(u^*) z=0.
\end{equation}
Since $-\frac{\partial f}{\partial u}(u^*)>0,$ then $z(t)\le 0$, by virtue of Lemma \ref{comp1}. Also, Eq. (\ref{ll}) holds true for $-z$ and thus
$-z \le 0, $  by virtue of Lemma \ref{comp1}. Thus, $z=0$ which proves that $u_1=u_2.$
\end{proof}

\begin{lemma}\label{nonlinear2}
Consider the nonlinear fractional differential equation
\begin{equation}\label{frac11}
({}^{ABC} {}_{a}D^\alpha u)(t)=f(t,u), \ \ t>a, \ 0<\alpha<1,
\end{equation}
 where $f(t,u)$ is a smooth function. Assume that
 $$\lambda_2 u+h_2(t) \le f(t,u)\le \lambda_1 u+h_1(t), \ \mbox{for all} \ t\in (a,b), u\in H^1(a,b),$$
 where $\lambda_1, \lambda_2 < 0.$ Let $v_1$ and $v_2$ be the solutions of
 \begin{equation}\label{frac22}
({}^{ABC} {}_{a}D^\alpha v_1)(t)=\lambda_1 v_1+h_1(t), \ \ t>0, \ 0<\alpha<1,
\end{equation}
and
\begin{equation}\label{frac33}
({}^{ABC} {}_{a}D^\alpha v_2)(t)=\lambda_2 v_2+h_2(t), \ \ t>0, \ 0<\alpha<1,
\end{equation}
  then it holds that $$v_2(t)\le u(t) \le v_1(t), \ t\ge a.$$
\end{lemma}
\begin{proof}

We shall prove that $u(t)\le v_1(t)$ and by applying analogous steps one can show that $v_2(t) \le u(t).$ By subtracting Eq. (\ref{frac22}) from Eq. (\ref{frac11}) we have
\begin{eqnarray}
\big({}^{ABC} {}_{a}D^\alpha (u-v_1)\big)(t)&=& f(t,u)-\lambda_1 v_1-h_1(t)\nonumber\\
&&\le \lambda_1 u+h_1(t)-\lambda_1 v_1-h_1(t)=\lambda_1(u-v_1).\nonumber
 \end{eqnarray}
 Let $z=u-v_1$, then it holds that
 $$({}^{ABC} {}_{a}D^\alpha z)(t)-\lambda_1 z(t) \le 0.$$
 Since $\lambda_1 > 0$, then $z\le 0,$ by virtue of Lemma \ref{comp1}, which proves the result.
\end{proof}

We now present some examples to illustrate  the efficiency of the obtained results.
\begin{example}
Consider the nonlinear fractional initial value problem
\begin{equation}\label{example1}
({}^{ABC} {}_{0}D^\alpha u)(t)=e^{-u}-2, \ t>0, \ 0<\alpha<1.
\end{equation}
$$u(0)=-\ln(2).$$
Since $e^{-u}-2\ge -u-1$, let
 $v$ be the solution of
\begin{equation}\label{mmm1}
({}^{ABC} {}_{0}D^\alpha v)(t)=-v-1, \ t>0, \ 0<\alpha<1,
\end{equation}
then $v(t)\le u(t)$ by virtue of Lemma \ref{nonlinear2}.  The solution of Eq. (\ref{mmm1}) is given by Eq. (\ref{sol}) with  $\lambda=-1,$ and $f(t)=-1.$ Thus,
$$u(t)\ge v(t)=-
\frac{1}{B(\alpha)+1-\alpha}\big(B(\alpha)  E_\alpha[wt^\alpha]+(1-\alpha)(E_\alpha[wt^\alpha]+\frac{\alpha}{1-\alpha} \frac{t^{\alpha-1}}{\Gamma(\alpha)}*E_\alpha[wt^\alpha]\big),$$
where $\omega=-\frac{\alpha}{B(\alpha)+1-\alpha}.$ We recall that Eq. (\ref{mmm1}) has a solution only if $v(0)=-
1.$
\end{example}

\begin{example}
Consider the nonlinear fractional initial value problem
\begin{equation}\label{example2}
({}^{ABC} {}_{0}D^\alpha u)(t)=e^{-u}-\frac{1}{2} u^2, \ t>0, \ 0<\alpha<1.
\end{equation}
$$u(0)=u_0,$$ where $u_0$ is the unique solution of $e^{-u_0}=\frac{1}{2} u_0^2.$ By the Taylor series expansion of $f(u)=e^{-u},$ one can easily show that
 $e^{-u}-\frac{1}{2} u^2 \le  1-u$.  Let
 $v$ be the solution of
\begin{equation}\label{mmm2}
({}^{ABC} {}_{0}D^\alpha v)(t)=-v+1, \ t>0, \ 0<\alpha<1,
\end{equation}
then $v(t)\ge u(t)$ by virtue of Lemma \ref{nonlinear2}.  The solution of Eq. (\ref{mmm2}) is given by Eq. (\ref{sol}) with $\lambda=-1,$ and $f(t)=1.$ Thus,
$$u(t)\le v(t)=
\frac{1}{B(\alpha)+1-\alpha}\big(B(\alpha)  E_\alpha[wt^\alpha]+(1-\alpha)(E_\alpha[wt^\alpha]+\frac{\alpha}{1-\alpha} \frac{t^{\alpha-1}}{\Gamma(\alpha)}*E_\alpha[wt^\alpha]\big),$$
where $\omega=-\frac{\alpha}{B(\alpha)+1-\alpha}.$ We recall that Eq. (\ref{mmm2}) has a solution only if $v(0)=1.$ Moreover, applying the result of Lemma \ref{big} we have
$||v|| \le 1$, and hence $||u|| \le 1.$
\end{example}

\begin{example}
Consider the nonlinear fractional initial value problem
\begin{equation}\label{example3}
({}^{ABC} {}_{0}D^\alpha u)(t)= -e^{u}(3+\cos(u))+ 4 e^{-t}, \ t>0, \ 0<\alpha<1.
\end{equation}
$$u(0)=0.$$ Let $h(u)=-e^{u}(3+\cos(u))$, since $h''(u)=e^u(-3+2\sin(u)) \le 0,$ by the Taylor series expansion method  one can easily show that
 $h(u) \le h(0)+h'(0) u= -4-4 u.$   Let
 $v$ be the solution of
\begin{eqnarray}\label{mmm3}
({}^{ABC} {}_{0}D^\alpha v)(t)&=&-4v-4+4 e^{-t}, \ t>0, \ 0<\alpha<1,\\
v(0)&=&0,\nonumber
\end{eqnarray}
then $v(t)\ge u(t)$ by virtue of Lemma \ref{nonlinear2}.  The solution of Eq. (\ref{mmm3}) is given by Eq. (\ref{sol}) where $\lambda=-4,$ and $f(t)=-4+4e^{-t}.$  Applying the result of Lemma \ref{big} we have
$$||u|| \le ||v|| \le |\frac{-4+4e^{-t}}{4}|=1-e^{-t}, \ t>0.$$
\end{example}


{\bf Acknowledgment}: The author acknowledges the support of the United Arab emirates University under the Fund No. 31S239-UPAR(1) 2016.

\end{document}